\documentclass[english,12pt,twoside]{article}
\usepackage{amssymb,amsbsy,amsmath,amsfonts,amssymb,amscd}
\usepackage{latexsym}
\usepackage{euscript}
\usepackage{exscale}
\usepackage{epsfig}
\usepackage{amsthm}
\usepackage[english]{babel}

 \newtheorem{thm}{Theorem}[section]
 \newtheorem{cor}[thm]{Corollary}
 \newtheorem{lem}[thm]{Lemma}
 \newtheorem{prop}[thm]{Proposition}
 \theoremstyle{definition}
 
 \theoremstyle{remark}

 \numberwithin{equation}{section}

\title{Compatibility of Riemannian structures and Jacobi structures}
\author{Yacine A\"it Amrane, Ahmed Zeglaoui}

\begin{document}

\maketitle

\begin{abstract}
\noindent We give a notion of compatibility between a Riemannian
structure and a Jacobi structure. We prove that in case of
fundamental examples of Jacobi structures : Poisson structures,
contact structures and locally conformally symplectic
 structures, we get respectively Riemann-Poisson structures
  in the sense of M. Boucetta, $\frac{1}{2}$-Kenmotsu structures and locally conformally K\"ahler structures.
\end{abstract}
\bigskip\bigskip

\hspace{-0.6cm}{\bf\Large Introduction}\medskip\\

Jacobi manifolds, introduced by A. Lichnerowicz \cite{lichnerowicz},
generalize Poisson manifolds, contact manifolds and locally
conformally symplectic manifolds. We ask the natural question of the
existence of a notion of compatibility between a Jacobi structure
and pseudo-Riemannian structure, a compatibility for which
particular Jacobi structures give arise to remarkable geometric
structures. In this work, we introduce such a notion which in the
case of a Poisson manifold gives a pseudo-Riemannian Poisson
structure in the sense of M. Boucetta. We prove that for a contact
Riemannian structure, with this notion of compatibility we get a
$\frac{1}{2}$-Kenmotsu structure, and that in the case of a locally
conformally symplectic structure with an "associated" metric, we get
a locally conformal K\"ahler structure.

Let $M$ be a smooth manifold. In this work, we consider on $M$ a
bivector field $\pi$, a vector field $\xi$ and a  $1$-form
$\lambda$, and associate with the triple $(\pi,\xi,\lambda)$ a skew
algebroid $(T^\ast M,\sharp_{\pi,\xi},[.,.]_{\pi,\xi}^\lambda)$ on
$M$. We prove that if the pair $(\pi,\xi)$ is a Jacobi structure and
that $\pi\neq 0$, the skew algebroid $(T^\ast
M,\sharp_{\pi,\xi},[.,.]_{\pi,\xi}^\lambda)$ is an almost Lie
algebroid if and only if $\sharp_{\pi,\xi}(\lambda)=\xi$. In case
$\xi=\lambda=0$, it is the cotangent algebroid of the Poisson
manifold $(M,\pi)$. We also prove that in case $(\pi,\xi)$ is the
Jacobi structure associated with a contact form $\eta$, respectively
with a locally conformally symplectic structure $(\omega,\theta)$,
the skew algebroid $(T^\ast
M,\sharp_{\pi,\xi},[.,.]_{\pi,\xi}^\eta)$, respectively $(T^\ast
M,\sharp_{\pi,\xi},[.,.]_{\pi,\xi}^\theta),$ is a Lie algebroid
isomorphic to the tangent algebroid of $M$.

Next, for a triple $(\pi,\xi,g)$ consisting of a bivector field
$\pi$, a vector field $\xi$ and a pseudo-Riemannian metric $g$ on
$M$, we put $\lambda=g(\xi,\xi)\flat_g(\xi)-\flat_g(J\xi)$ and
$[.,.]_{\pi,\xi}^g=[.,.]_{\pi,\xi}^\lambda$, where
$\flat_g:TM\rightarrow T^\ast M$ and $\sharp_g=\flat_g^{-1}$ are the
musical isomorphisms of $g$ and where $J$ is the endomorphism of the
tangent bundle $TM$ given by
$\pi(\alpha,\beta)=g(J\sharp_g(\alpha),\sharp_g(\beta))$, and define
a contravariant derivative $\mathcal{D}$ to be the unique
contravariant symmetric derivative compatible with $g$. If
$(\pi,\xi)$ is Jacobi, and if $\sharp_{\pi,\xi}$ is an isometry, a
condition that is   satisfied in the particular cases of a contact
and of a locally conformally symplectic structure, we prove that
$\mathcal{D}$ is related to the (covariant) Levi-Civita connection
$\nabla$ of $g$ by
$\sharp_{\pi,\xi}(\mathcal{D}_\alpha\beta)=\nabla_{\sharp_{\pi,\xi}(\alpha)}\sharp_{\pi,\xi}(\beta)$.

Finally, with the use of the contravariant Levi-Civita derivative
$\mathcal{D}$ we introduce a notion of compatibility of the triple
$(\pi,\xi,g)$. In case $\xi=0$, it is just the compatibility of the
pair $(\pi,g)$ introduced by M. Boucetta (\cite{boucetta1}). In the
case of a Jacobi structure $(\pi,\xi)$ associated with a contact
metric structure $(\eta,g)$, the triple $(\pi,\xi,g)$ is compatible
if and only if the structure $(\eta,g)$ is $\frac{1}{2}$-Kenmotsu.
In case $(\pi,\xi)$ is the Jacobi structure associated with a
locally conformally symplectic structure $(\omega,\theta)$, if $g$
is  a somehow associated metric, the triple $(\pi,\xi,g)$ is
compatible if and only if the structure $(\omega,\theta,g)$ is
locally conformally K\"ahler.

\section{Almost Lie algebroids associated with a Jacobi manifold}

\subsection{Almost Lie algebroids associated with a Jacobi manifold}

Throughout this paper $M$ is a smooth manifold, $\pi$ a bivector
field and $\xi$ a vector field on $M$.

The pair $(\pi,\xi)$ defines a Jacobi structure on $M$ if we have
the relations
\begin{equation}\label{jacobi}
\begin{array}{ccc}
\left[ \pi ,\pi \right] =2\xi \wedge \pi & \text{ \ \ et \ \ } &
\left[ \xi
,\pi \right] :=\mathcal{L}_{\xi }\pi =0,%
\end{array}
\end{equation}
where $\left[ .,.\right]$ is the Schouten-Nijenhuis bracket. We say
that $(M,\pi,\xi)$ is a Jacobi manifold. In the case $\xi =0$, the
relations above are reduced to $\left[ \pi ,\pi \right]=0$ that
corresponds to a Poisson structure $(M,\pi)$.

Recall on the other hand, see for instance \cite{grabowski}, that a
skew algebroid over $M$ is a triple $\left( E,\sharp
_{E},\left[.,.\right]_{E}\right)$ where $E$ is the total space of a
vector bundle on $M$, $\sharp_{E}$ is a vector bundle morphism from
$E$ to $TM$, called the anchor map, and $\left[ .,.\right]_{E}:
\Gamma(E) \times \Gamma(E)\longrightarrow \Gamma(E)$,
$(s,t)\longmapsto\left[ s,t\right]_{E}$, is an alternating
$\mathbb{R}$-bilinear map over the space $\Gamma(E)$ of sections of
$E$ verifying the Leibniz identity :
\begin{equation*}
\left[ s,\varphi t\right] _{E}=\varphi \left[ s,t\right] _{E}+
\sharp_{E}(s)(\varphi) t, \quad \forall \varphi \in C^{\infty}(M),\;
\forall s,t\in \Gamma(E).
\end{equation*}
A skew algebroid $\left(E,\sharp_{E},\left[.,.\right]_{E}\right)$ is
an almost Lie algebroid if
\begin{equation*}
\sharp_{E}\left( \left[ s,t\right]_{E}\right) =\left[
\sharp_{E}(s),\sharp_{E}(t) \right] , \quad \forall s,t\in
\Gamma(E),
\end{equation*}
and a Lie algebroid if $\left(\Gamma(E) ,\left[
.,.\right]_{E}\right) $ is a Lie algebra, i.e., if
\begin{equation*}
\left[ s,\left[ t,r\right] _{E}\right] _{E}+\left[ t,\left[ r,s\right] _{E}%
\right] _{E}+\left[ r,\left[ s,t\right] _{E}\right] _{E}=0, \quad
\forall  s,t,r\in \Gamma(E).
\end{equation*}
Note that a Lie algebroid is an almost Lie algebroid and that, on
the other hand, an almost Lie algebroid $\left( E,\sharp_{E},\left[
.,.\right]_{E}\right)$ such that the anchor map $\sharp _{E}$ is an
isomorphism is a Lie algebroid isomorphic to the tangent algebroid
$(TM,\mathrm{id}_M,[.,.])$ of $M$.

Let $\sharp_{\pi }: T^\ast M \longrightarrow TM$  be the vector
bundle morphism defined by $ \beta\left(\sharp_{\pi }\left( \alpha
\right)\right)=\pi\left( \alpha ,\beta\right)$ and let $\left[
.,.\right]_{\pi }:\Omega^{1}(M) \times \Omega ^{1}(M)
\longrightarrow \Omega^{1}(M)$ be the map defined by
\begin{equation*}\label{pi-koszul}
\left[ \alpha ,\beta \right] _{\pi }:=\mathcal{L}_{\sharp_{\pi
}(\alpha)}\beta -\mathcal{L}_{\sharp_{\pi}(\beta)}\alpha -d\left(
\pi(\alpha,\beta) \right),
\end{equation*}
called the Koszul bracket. Consider the morphism of vector bundles
$\sharp_{\pi,\xi}:  T^{\ast}M \longrightarrow TM$ defined by
\begin{equation*}
\sharp_{\pi,\xi}(\alpha)=\sharp_{\pi}(\alpha)+\alpha(\xi) \xi
\end{equation*}
and, for a $1$-form $\lambda \in \Omega ^{1}(M)$, the map $\left[
.,.\right]_{\pi ,\xi }^{\lambda}:  \Omega ^{1}(M)\times \Omega
^{1}(M)\longrightarrow \Omega ^{1}(M)$ defined by
\begin{equation*}
\left[ \alpha ,\beta \right]_{\pi,\xi}^{\lambda}:=\left[
\alpha,\beta \right]_{\pi }+\alpha(\xi) \left(\mathcal{L}_{\xi}\beta
-\beta\right)-\beta(\xi) \left(\mathcal{L}_{\xi}\alpha-\alpha\right)
-\pi(\alpha,\beta) \lambda.
\end{equation*}
The triple
$(T^{\ast}M,\sharp_{\pi,\xi},\left[.,.\right]_{\pi,\xi}^{\lambda})$
associated with $(\pi,\xi,\lambda)$ is a skew algebroid on $M$.

In case $\xi =\lambda =0$, the triple $(T^{\ast}M,\sharp_{\pi,\xi
},\left[ .,.\right]_{\pi,\xi }^{\lambda})$ is just the skew
algebroid $(T^{\ast}M,\sharp_{\pi},\left[ .,.\right]_{\pi})$
associated with the bivector field $\pi$. Recall that for any
differential forms $\alpha ,\beta ,\gamma \in \Omega ^{1}(M)$ we
have
\begin{equation}\label{poisson-pre-alg}
\gamma\left(\sharp_\pi\left(\left[\alpha,\beta\right]_\pi\right)-\left[\sharp_\pi(\alpha),\sharp_\pi(\beta)\right]\right)
=\dfrac{1}{2}\left[\pi,\pi\right]\left(\alpha ,\beta,\gamma \right),
\end{equation}
and for any $\varphi,\psi,\phi\in C^\infty(M)$ we have
\begin{equation*}\label{poisson-alg}
\left[d\varphi,\left[d\psi,d\phi\right]_{\pi}\right]_{\pi}+\left[d\psi,\left[
d\phi,d\varphi\right]_{\pi}\right]_{\pi}+\left[d\phi,\left[d\varphi,d\psi
\right]_{\pi}\right]_{\pi}\!=\!-\dfrac{1}{2}d\left(\left[\pi,\pi\right]\left(d\varphi
,d\psi,d\phi \right)\right).
\end{equation*}
Thus, $(T^{\ast}M,\sharp_{\pi},\left[.,.\right]_{\pi })$ is a Lie
algebroid if and only if $\pi$ is a Poisson tensor. If $\pi$ is a
Poisson tensor on $M$, the triple $\left( T^{\ast }M,\sharp _{\pi
},\left[ .,.\right]_{\pi}\right)$ is called the cotangent algebroid
of the Poisson manifold $(M,\pi)$. In the  case of a Jacobi
structure we have the following result

\begin{thm}\label{torsion-pre-alg}
Assume that $(\pi,\xi)$ is a Jacobi structure on $M$ and let
$\lambda\in \Omega^1(M)$. We have
$$
\sharp_{\pi,\xi}(\left[\alpha,\beta\right]_{\pi,\xi}^\lambda)-
\left[\sharp_{\pi,\xi}(\alpha),
\sharp_{\pi,\xi}(\beta)\right]=\pi(\alpha,\beta)\left(\xi-\sharp_{\pi,\xi}(\lambda)\right),
$$
for any differential forms $\alpha,\beta\in \Omega^1(M)$.
\end{thm}

\begin{proof} We have
$$
\begin{array}{lll}
\sharp_{\pi,\xi}(\left[ \alpha ,\beta \right]_{\pi,\xi}^{\lambda})
&=&\sharp_{\pi}\left( \left[ \alpha ,\beta \right]_{\pi}\right)-\alpha(\xi)\sharp_\pi(\beta)+\beta(\xi)\sharp_\pi(\alpha)-\pi(\alpha,\beta)\sharp_{\pi,\xi}(\lambda) \\
&& +\alpha(\xi)\sharp_\pi(\mathcal{L}_\xi\beta)
-\beta(\xi)\sharp_\pi(\mathcal{L}_\xi\alpha)+\left[\sharp_{\pi,\xi}(\alpha)(\beta(\xi))\right. \\
&& \left. -\sharp_{\pi,\xi}(\beta)(\alpha(\xi))-\xi(\pi(\alpha,\beta))+\beta(\mathcal{L}_\xi(\sharp_\pi(\alpha))) \right. \\
&& \left. -\alpha(\mathcal{L}_\xi(\sharp_\pi(\beta)))\right]\xi
\end{array}
$$
and on the other hand
$$
\begin{array}{lll}
\left[ \sharp_{\pi,\xi}(\alpha),\sharp_{\pi,\xi}(\beta) \right]
&=&\left[\sharp_{\pi}(\alpha),\sharp_{\pi}(\beta) \right]
+\alpha(\xi)\mathcal{L}_\xi(\sharp_{\pi}(\beta)) -\beta(\xi)
\mathcal{L}_\xi(\sharp_{\pi}(\alpha))  \\
&&+\left[\sharp_{\pi,\xi}(\alpha)(\beta(\xi)) -
\sharp_{\pi,\xi}(\beta)(\alpha(\xi)) \right] \xi .
\end{array}
$$
Therefore, using the identity (\ref{poisson-pre-alg}), we deduce
that
$$
\begin{array}{lll}
\sharp_{\pi,\xi}( \left[ \alpha ,\beta \right]_{\pi,\xi}^{\lambda})
-\left[\sharp_{\pi,\xi}(\alpha) ,\sharp_{\pi,\xi}(\beta)\right]\!
&\!=\!&\!\left(\frac{1}{2}\left[\pi,\pi\right]-\xi\wedge
\pi\right)(\alpha,\beta,\cdot) \\
&& -\alpha(\xi)\mathcal{L}_\xi\sharp_\pi(\beta)
+\beta(\xi)\mathcal{L}_\xi\sharp_\pi(\alpha) \\
&& -\left[\alpha(\mathcal{L}_\xi\sharp_\pi(\beta))-\beta(\mathcal{L}_\xi\sharp_\pi(\alpha))\right. \\
&& \left. +\mathcal{L}_\xi\pi(\alpha,\beta)\right]\xi
 +\pi(\alpha,\beta)(\xi-\sharp_{\pi,\xi}(\lambda)).
\end{array}
$$
Now use the relations (\ref{jacobi}).
\end{proof}

\begin{cor}\label{lambda-pre-alg}
Assume that $(\pi,\xi)$ is a Jacobi structure on $M$ and let
$\lambda\in\Omega^1(M)$. If $\sharp_{\pi,\xi}(\lambda) =\xi$, the
skew algebroid $(T^{\ast}M,\sharp_{\pi ,\xi },\left[ .,.\right]_{\pi
,\xi}^{\lambda})$ associated with the triple $(\pi,\xi ,\lambda)$ is
an almost Lie algebroid, i.e.
$$
\sharp_{\pi,\xi}(\left[\alpha,\beta\right]_{\pi,\xi}^\lambda)=
\left[\sharp_{\pi,\xi}(\alpha), \sharp_{\pi,\xi}(\beta)\right],
$$
for any differential forms $\alpha ,\beta \in \Omega ^{1}(M)$. The
converse is also true if $\pi\neq 0$.
\end{cor}
\begin{proof}
It is a direct consequence of the theorem above.
\end{proof}

\subsection{Cotangent algebroid of a contact manifold}

Assume that $M$ is of odd dimension $2n+1$, $n\in \mathbb{N}^\ast$.
Recall that a contact form on $M$ is a differential $1$-form $\eta$
on $M$ such that the form $\eta \wedge (d\eta)^{\wedge^{n}}$ is a
volume form. Assume that the pair $(\pi,\xi)$ is the Jacobi
structure associated with a contact form $\eta$ on $M$, i.e. we have
$$
\pi(\alpha,\beta)=d\eta\left(\sharp_{\eta }(\alpha)
,\sharp_{\eta}(\beta)\right),
$$
where $\sharp_{\eta }$ is the inverse isomorphism of the isomorphism
of vector bundles $\flat_{\eta }:  TM \rightarrow T^\ast M$,
$\flat_\eta(X)=-i_{X}d\eta +\eta(X) \eta$, and
$\xi=\sharp_\eta(\eta)$. The vector field $\xi$ is called the Reeb
field associated with the contact structure $(M,\eta)$, it is
characterized by the formulae
\begin{equation*}
\begin{array}{ccc}
i_{\xi}d\eta :=d\eta( \xi ,.) =0 & \text{ \ \ and \ \ } &
i_{\xi}\eta :=\eta(\xi) =1.
\end{array}%
\end{equation*}

\begin{prop}\label{eta-alg}
The skew algebroid $(T^{\ast}M,\sharp_{\pi ,\xi},\left[
.,.\right]_{\pi,\xi }^{\eta})$ is a Lie algebroid isomorphic to the
tangent algebroid of $M$.
\end{prop}
\begin{proof}
Let us show that $\sharp_{\pi,\xi}$ is equal to the isomorphism
$\sharp_\eta$, inverse of the isomorphism $\flat_\eta$. Let
$\alpha,\beta \in \Omega^1(M)$, and let $X,Y$ be such that
$\alpha=\flat_\eta(X)$ and $\beta=\flat_\eta(Y)$. First, notice that
$\alpha(\xi)=\flat_\eta(X)(\xi)=\eta(X)$ and $\beta(\xi)=\eta(Y)$.
Therefore, we have
$$
\begin{array}{ll}
\beta(\sharp_{\pi,\xi}(\alpha)) & =\pi(\alpha,\beta)+ \eta(X)\eta(Y)\\
& =(-i_Yd\eta+\eta(Y)\eta)(X) \\
& =\flat_\eta(Y)(X) \\
& =\beta(\sharp_\eta(\alpha)).
\end{array}
$$
Thus $\sharp_{\pi,\xi}=\sharp_{\eta}$ and in particular
$\sharp_{\pi,\xi}(\eta)=\sharp_{\eta}(\eta)=\xi$. We see that the
proposition is a consequence of Corollary \ref{lambda-pre-alg} and
the fact that $\sharp_{\pi,\xi}$ is an isomorphism.
\end{proof}

Hence, if $(M,\eta)$ is a contact manifold and if $(\pi,\xi)$ is the
associated Jacobi structure, by the proposition above we have
$\sharp_{\pi,\xi}=\sharp_\eta$. If we put $\left[
.,.\right]_\eta=\left[ .,.\right]_{\pi,\xi}^\eta$, then we have a
Lie algebroid $(T^\ast M, \sharp_\eta,\left[ .,.\right]_\eta)$
associated naturally with the contact manifold $(M,\eta)$. This Lie
algebroid may henceforth be called the cotangent algebroid of the
contact manifold $(M,\eta)$.

\subsection{Cotangent algebroid of a locally conformally symplectic manifold}

A locally conformally symplectic structure on $M$ is a pair
$(\omega,\theta)$ of a differential closed $1$-form $\theta$ and a
nondegenerate  differential $2$-form $\omega$ on $M$ such that
\begin{equation*}
d\omega +\theta \wedge \omega =0.
\end{equation*}
In the particular case where $\theta$ is exact, i.e. $\theta =df$,
we say that $(\omega ,df)$ is conformally symplectic, it is
equivalent to $e^{f}\omega$ being symplectic and this justifies the
terminology.

The next proposition shows that it is equivalent to give a locally
conformally symplectic manifold and to give a Jacobi manifold such
that the underlying bivector field is nondegenerate (see also
\cite[\S 2.3, ex. 4]{marle}). Having not found a proof in the
literature, we give one here.

Assume that $\omega\in\Omega^2(M)$ is a nondegenerate $2$-form and
let $\theta\in \Omega^1(M)$. Assume that the pair $(\pi,\xi)$ is
associated with the pair $(\omega,\theta)$, i.e. $
i_{\sharp_\pi(\alpha)}\omega=-\alpha$ for any
$\alpha\in\Omega^1(M)$, and that $i_\xi\omega = -\theta$. This means
that
$$
\pi(\alpha,\beta)=\omega(\sharp_\omega(\alpha),\sharp_\omega(\beta))
$$
where $\sharp_\omega$ is the inverse isomorphism of the vector
bundle isomorphism $\flat_\omega:TM \longrightarrow T^\ast M$,
$\flat_\omega(X)=-i_X\omega$, and that $\xi=\sharp_\omega(\theta)$.
We have the following

\begin{lem}\label{omega-theta-pi-xi}
Let $X,Y,Z$ be vector fields on $M$ and let $\alpha,\beta,\gamma$ be
the differential $1$-forms such that $X=\sharp_\pi(\alpha)$,
$Y=\sharp_\pi(\beta)$ and $Z=\sharp_\pi(\gamma)$. We have
\begin{enumerate}
\item $(dw + \theta \wedge
\omega)(X,Y,Z)=\left(\frac{1}{2}\left[\pi,\pi\right]-\xi\wedge
\pi\right)(\alpha,\beta,\gamma)$.
\item $\mathcal{L}_\xi\omega(X,Y)=-\mathcal{L}_\xi\pi(\alpha,\beta)$.
\end{enumerate}
\end{lem}
\begin{proof} Using the identity $\pi(\alpha,\beta)=\omega(X,Y)$
and the identity (\ref{poisson-pre-alg}), we get
$$
\omega([X,Y],Z)=\gamma([X,Y])=-\frac{1}{2}\left[\pi,\pi\right](\alpha,\beta,\gamma)+\pi(\left[\alpha,\beta\right]_\pi,\gamma),
$$
therefore, with a direct calculation, we deduce
\begin{equation}\label{domega}
d\omega(X,Y,Z)=\frac{1}{2}\left[\pi,\pi\right](\alpha,\beta,\gamma).
\end{equation}
On the other hand, notice that
$\theta(X)=-i_\xi\omega(X)=i_X\omega(\xi)=i_{\sharp_\pi(\alpha)}\omega(\xi)=-\alpha(
\xi)$, likewise $\theta(Y)=-\beta(\xi)$ and
$\theta(Z)=-\gamma(\xi)$, thus
$\theta\wedge\omega(X,Y,Z)=-\xi\wedge\pi(\alpha,\beta,\gamma)$.
Hence, with (\ref{domega}), we get the first assertion of the lemma.
For the second assertion, is suffices to notice that
$$
\pi(\mathcal{L}_\xi\alpha,\beta)=-\mathcal{L}_\xi\alpha(Y)
=-\xi(\alpha(Y))+\alpha(\mathcal{L}_{\xi}Y)=\xi(\omega(X,Y))-\omega(X,\mathcal{L}_\xi
Y).
$$
\end{proof}

\begin{prop}\label{loc-conf-sympl-jac}
The pair $(\omega,\theta)$ is a locally conformally symplectic
structure if and only if the pair $(\pi,\xi)$ is a Jacobi structure.
\end{prop}
\begin{proof}
From the first assertion of Lemma \ref{omega-theta-pi-xi} we deduce
that the identity $d\omega+\theta \wedge \omega =0$ is satisfied if
and only if  the identity $\left[ \pi ,\pi \right] =2\xi \wedge \pi$
is, and if one of the two is satisfied then, using the Cartan
formula, we get
$$
\mathcal{L}_\xi\omega  =d(i_\xi \omega)+i_\xi d\omega =-d\theta -
i_\xi(\theta\wedge\omega) =-d\theta,
$$
and then, with the assertion 2. of Lemma \ref{omega-theta-pi-xi},
that $\mathcal{L}_\xi\pi=0$ if and only if $d\theta=0$.
\end{proof}

\begin{prop}\label{omega-theta-alg}
Assume that $(M,\omega,\theta)$ is a locally conformally symplectic
manifold and let $(\pi,\xi)$ be the associated Jacobi structure. The
skew algebroid $(T^{\ast}M,\sharp_{\pi,\xi},\left[
.,.\right]_{\pi,\xi }^{\theta})$ is a Lie algebroid isomorphic to
the tangent algebroid of $M$.
\end{prop}
\begin{proof}
Since
$\sharp_{\pi,\xi}(\theta)=\sharp_\pi(\theta)+\theta(\xi)\xi=\sharp_\pi(\theta)=\xi$.
Then, by Corollary \ref{lambda-pre-alg}, the triple
$(T^{\ast}M,\sharp_{\pi ,\xi},\left[ .,.\right]_{\pi,\xi
}^{\theta})$ is an almost Lie algebroid. It remains to prove that
$\sharp_{\pi,\xi}$ is an isomorphism. It suffices to prove that
 it is injective. Since we have
 $\sharp_{\pi,\xi}(\alpha)=\sharp_\pi(\alpha+\alpha(\xi)\theta)$ and
 that the bivector field $\pi$ is nondegenerate, then
 $\sharp_{\pi,\xi}(\alpha)=0$ implies $\alpha=-\alpha(\xi)\theta$,
 thus $\alpha(\xi)=-\alpha(\xi)\theta(\xi)=0$, and therefore $\alpha=0$.
\end{proof}

Hence, if $(M,\omega,\theta)$ is a locally conformally symplectic
manifold and $(\pi,\xi)$ the associated Jacobi structure, by the
above proposition, if we put
$\sharp_{\omega,\theta}:=\sharp_{\pi,\xi}$ and $\left[
.,.\right]_{\omega,\theta}:=\left[ .,.\right]_{\pi,\xi}^\theta$,
then we have a Lie algebroid $(T^\ast M,
\sharp_{\omega,\theta},\left[ .,.\right]_{\omega,\theta})$
associated naturally with the locally conformally symplectic
manifold $(M,\omega,\theta)$. This Lie algebroid may be called the
cotangent algebroid of the locally conformally symplectic manifold
$(M,\omega,\theta)$.

\section{Levi-Civita contravariant derivative associated with the
triple $(\pi,\xi,g)$}

\subsection{Definition and properties}

In all what follows, we denote by $g$ a pseudo-Riemannian metric on
$M$, by $\flat_g:TM\rightarrow T^\ast M$ the vector bundle
isomorphism such that $\flat_g(X)(Y)=g(X,Y)$, by $\sharp_{g}$ the
inverse isomorphism of $\flat_g$, and by $g^\ast$ the cometric of
$g$, i.e. the tensor field defined by $ g^\ast(\alpha,\beta)
:=g\left( \sharp_{g}(\alpha) ,\sharp_{g}(\beta) \right)$.

With the pair $(\pi,g)$ we associate the vector field endomorphisms
$J$ of $TM$ and $J^\ast$ of $T^\ast M$ defined respectively by
\begin{equation}\label{J}
g(J\sharp_{g}(\alpha),\sharp_{g}(\beta))=\pi(\alpha,\beta)
\quad\textrm{ and }\quad g^\ast(J^\ast \alpha ,\beta)
=\pi(\alpha,\beta).
\end{equation}
We have $J=\sharp_g\circ J^\ast \circ \flat_g$. With the triple
$(\pi,\xi,g)$ we associate the differential $1$-form $\lambda$
defined by
\begin{equation*}
\lambda =g(\xi,\xi)\flat_{g}(\xi)-\flat_g(J\xi),
\end{equation*}
and we use the notation $\left[.,.\right] _{\pi ,\xi }^{g}$ instead
of $\left[ .,.\right]_{\pi,\xi}^{\lambda}$.

We call the contravariant Levi-Civita derivative associated with the
triple $(\pi,\xi,g)$ the unique derivative $\mathcal{D}: \Omega
^{1}(M)\times \Omega ^{1}(M) \longrightarrow \Omega ^{1}(M)$,
symmetric with respect to the bracket $\left[.,.\right] _{\pi ,\xi
}^{g}$ and compatible with the metric. It is entirely characterized
by the formula :
\begin{equation}\label{formule-koszul-jacobi}
\begin{array}{lll}
2g^\ast\left( \mathcal{D}_{\alpha }\beta,\gamma \right)
&=&\sharp_{\pi,\xi}(\alpha)\cdot g^\ast(\beta,\gamma)
+\sharp_{\pi,\xi}(\beta)\cdot
g^\ast(\alpha,\gamma)-\sharp_{\pi,\xi}(\gamma)\cdot g^\ast(\alpha,\beta) \\
&&-g^\ast(\left[\beta,\gamma\right]_{\pi,\xi }^{g},\alpha)
-g^\ast(\left[\alpha,\gamma \right]_{\pi,\xi }^{g},\beta)
+g^\ast(\left[\alpha,\beta\right]_{\pi,\xi }^{g},\gamma).
\end{array}
\end{equation}
In the case $\xi =0$, the derivative $\mathcal{D}$ is just the
Levi-Civita contravariant derivative associated in \cite{boucetta1}
with the pair $(\pi,g)$.

\begin{prop}\label{Levi-Civita-triplet}
Assume that the skew algebroid $(T^\ast M,\sharp_{\pi,\xi},\left[
.,.\right] _{\pi ,\xi }^{g})$ is an almost Lie algebroid  and that
the anchor map $\sharp_{\pi,\xi}$ is an isometry. Then
$$
\sharp_{\pi,\xi}\left(\mathcal{D}_\alpha\beta\right)
=\nabla_{\sharp_{\pi,\xi}(\alpha)}\sharp_{\pi,\xi}(\beta),
$$
where $\nabla$ is the Levi-Civita (covariant) connection associated
with $g$.
\end{prop}
\begin{proof}
Since we have assumed that $(T^\ast M,\sharp_{\pi,\xi},\left[
.,.\right] _{\pi ,\xi }^{g})$ is an almost Lie algebroid, we have
$$
\sharp_{\pi,\xi}(\left[\alpha,\beta\right]_{\pi,\xi}^g)=\left[\sharp_{\pi,\xi}(\alpha),\sharp_{\pi,\xi}(\beta)\right],
$$
for every $\alpha,\beta \in \Omega^1(M)$. Since we have also assumed
that $\sharp_{\pi,\xi}$ is an isometry, from the formula
(\ref{formule-koszul-jacobi}) and the Koszul formula relative to the
Levi-Civita connection $\nabla$ of $g$ we deduce that
$$
g^\ast\left( \sharp_{\pi,\xi}(\mathcal{D}_{\alpha
}\beta),\sharp_{\pi,\xi}(\gamma) \right)= g\left(
\nabla_{\sharp_{\pi,\xi}(\alpha)
}\sharp_{\pi,\xi}(\beta),\sharp_{\pi,\xi}(\gamma) \right)
$$
for any differential $1$-forms $\alpha,\beta,\gamma \in
\Omega^1(M)$.
\end{proof}

\subsection{Skew algebroid  associated with an almost contact Riemannian manifold}

Let $(\Phi,\xi,\eta)$ be a triple consisting of a $1$-form $\eta$, a
vector field $\xi$ and $(1,1)$-tensor field $\Phi$ on $M$. The
triple $(\Phi,\xi,\eta)$ defines an almost contact structure on $M$
if $\Phi^{2}=-\textrm{Id}_{TM}+\eta \otimes \xi$ and $\eta(\xi)=1$.
From what we deduce, see for instance \cite[Th. 4.1]{blair}, that
$\Phi(\xi)=0$ and $\eta \circ \Phi =0$.

We say that the metric $g$ is associated with the triple
$(\Phi,\xi,\eta)$ if the following identity is verified
\begin{equation}\label{met-ass-presque-cont}
g\left(\Phi(X),\Phi(Y)\right)=g(X,Y)-\eta(X)\eta(Y).
\end{equation}
We say that the manifold $(M,\Phi,\xi,\eta,g)$ is almost contact
pseudo-Riemannian if the triple $(\Phi,\xi,\eta)$ is an almost
contact structure and $g$ an associated metric. If in addition the
metric $g$ is positive definite, we say that $(M,\Phi,\xi,\eta,g)$
is an almost contact Riemannian manifold. Notice that if we set
$Y=\xi$ in the formula (\ref{met-ass-presque-cont}), we deduce that
if $(\Phi,\xi,\eta,g)$ is an almost contact pseudo-Riemannian
structure then

\begin{equation*}
g(X,\xi)=\eta(X),
\end{equation*}
for any $X\in \mathfrak{X}(M)$, i.e. $\flat_g(\xi)=\eta$. In
particular $g(\xi,\xi)=1$.

\begin{prop}\label{pi-ass-presque-contact}
If $(\Phi,\xi,\eta,g)$ is an almost contact pseudo-Riemannian
structure on $M$, then the map $\pi:\Omega^{1}(M)\times
\Omega^{1}(M)\rightarrow C^\infty(M)$ defined by
$$
\pi(\alpha,\beta)=g(\sharp_g(\alpha),\Phi(\sharp_g(\beta)))
$$
is a bivector field on $M$ and the vector bundle morphism
$\sharp_{\pi,\xi}$ is an isometry.
\end{prop}

\begin{proof}
Let $\alpha\in\Omega^1(M)$ and put $X=\sharp_g(\alpha)$. Using
(\ref{met-ass-presque-cont}) and $\eta\circ\Phi=0$, we get
$\pi(\alpha,\alpha)= g(\Phi(X),\Phi^2(X))$, and since
$\Phi^2=-\textrm{Id}_{TM}+\eta\otimes\xi$ and
$g(\Phi(X),\xi)=\eta\circ\Phi(X)=0$, we get
$$
\pi(\alpha,\alpha)=-g(\Phi(X),X)+\eta(X)g(\Phi(X),\xi) =
-g(\Phi(X),X) = -\pi(\alpha,\alpha),
 $$
and then, that $\pi$ is a bivector field. Let us prove that
$\sharp_{\pi,\xi}$ is an isometry. Let $\alpha\in \Omega^1(M)$.
 Recall that by definition, we have
 $\sharp_{\pi,\xi}(\alpha)=\sharp_\pi(\alpha)+\alpha(\xi)\xi$.
 Since we have on one hand
 $\alpha(\xi)=g(\sharp_g(\alpha),\xi)=\eta(\sharp_g(\alpha))$ and on the
 other hand, for any $\beta\in \Omega^1(M)$,
 $$
\beta(\sharp_\pi(\alpha))=g(\sharp_g(\alpha),\Phi(\sharp_g(\beta)))
=-g(\Phi(\sharp_g(\alpha)),\sharp_g(\beta))=-\beta(\Phi(\sharp_g(\alpha))),
 $$
i.e. $\sharp_\pi(\alpha)=-\Phi(\sharp_g(\alpha))$, we deduce that
\begin{equation}\label{f1}
\sharp_{\pi,\xi}(\alpha)=-\Phi(\sharp_g(\alpha))+\eta(\sharp_g(\alpha))\xi.
\end{equation}
Let $\alpha,\beta\in \Omega^1(M)$. From the formula (\ref{f1}) and
the fact that $g(\Phi(X),\xi)=\eta\circ\Phi(X)=0$ and
$g(\xi,\xi)=1$, we deduce that
$$
g(\sharp_{\pi,\xi}(\alpha),\sharp_{\pi,\xi}(\beta))
=g(\Phi(\sharp_g(\alpha)),\Phi(\sharp_g(\beta)))
+\eta(\sharp_g(\alpha))\eta(\sharp_g(\beta)).
$$
By using the formula (\ref{met-ass-presque-cont}), we get $
g(\sharp_{\pi,\xi}(\alpha),\sharp_{\pi,\xi}(\beta))
=g(\sharp_g(\alpha),\sharp_g(\beta))=g^\ast(\alpha,\beta)$.
\end{proof}

\begin{cor}\label{riem-presque-contact-levi-civita}
Assume $(\Phi,\xi,\eta,g)$ is an almost contact pseudo-Riemannian
structure on $M$ and let $\pi$ be the associated bivector field,
i.e. the one defined in Proposition \ref{pi-ass-presque-contact}. If
the skew algebroid $(T^\ast M,\sharp_{\pi,\xi},\left[.,.\right]
_{\pi ,\xi }^{g})$ is an almost Lie algebroid, then
\begin{equation*}
\sharp_{\pi,\xi}\left( \mathcal{D}_{\alpha }\beta \right)
=\mathcal{\nabla}_{\sharp_{\pi,\xi}(\alpha)}\sharp_{\pi,\xi}(\beta),
\end{equation*}
for every $\alpha ,\beta \in \Omega ^{1}(M)$.
\end{cor}
\begin{proof}
This is a direct consequence of Propositions
\ref{pi-ass-presque-contact} and \ref{Levi-Civita-triplet}.
\end{proof}

Assume that $\eta$ is a contact form on $M$. We say that the
manifold $(M,\eta,g)$ is contact pseudo-Riemannian, or that the
metric $g$ is associated with the contact form $\eta$, if there
exists a vector field endomorphism $\Phi$ of $TM$ such that
$(\Phi,\xi,\eta,g)$ is an almost contact pseudo-Riemannian structure
and that
\begin{equation}\label{met-ass-cont}
g(X,\Phi(Y))=d\eta(X,Y).
\end{equation}
If in addition $g$ is positive definite, we say that $(M,\eta,g)$ is
a contact Riemannian manifold.

\begin{thm}\label{riemannienne-contact-levi-civita}
Assume that $(M,\eta,g)$ is a contact pseudo-Riemannian manifold.
 We have
\begin{equation*}
\sharp_{\eta}\left( \mathcal{D}_{\alpha }\beta \right)
=\mathcal{\nabla}_{\sharp_{\eta}(\alpha)}\sharp_{\eta}(\beta),
\end{equation*}
for every $\alpha ,\beta \in \Omega ^{1}(M)$.
\end{thm}
\begin{proof}
Let $(\Phi,\xi,\eta,g)$ be the almost contact pseudo-Riemannian
structure associated with the contact pseudo-Riemannian manifold
$(M,\eta,g)$. Let $(\pi,\xi)$ be the Jacobi structure associated
with $\eta$, then $\sharp_\eta=\sharp_{\pi,\xi}$. By Proposition
\ref{eta-alg} and the corollary above, we need only to prove that
$\pi$ is associated with $(\Phi,\xi,\eta,g)$ and that
$\eta=\lambda$. Let $\alpha \in \Omega^1(M)$ and put
$X=\sharp_\eta(\alpha)$. By using (\ref{met-ass-cont}), we have
$$
\sharp_g(\alpha)=\sharp_g(\flat_\eta(X))=-\sharp_g(i_X
d\eta)+\eta(X)\xi = \Phi(X)+\eta(X)\xi.
$$
Therefore, applying $\Phi$,
$$
\Phi(\sharp_g(\alpha))=\Phi^2(X)=-X+\eta(X)\xi=-\sharp_\eta(\alpha)+\alpha(\xi)\xi=-\sharp_\pi(\alpha).
$$
We deduce that
$\pi(\alpha,\beta)=g(\sharp_g(\alpha),\Phi(\sharp_g(\beta)))$ for
any $\alpha,\beta\in \Omega^1(M)$, and that $\Phi=-J$, where $J$ is
the field of endomorphisms associated with the pair $(\pi,g)$.
Hence, $J\xi=0$, and since $g(\xi,\xi)=1$ it follows that
$\lambda=\flat_g(\xi)=\eta$.
\end{proof}

\subsection{Riemannian metric associated with a locally conformally symplectic structure}

Assume that $\omega\in \Omega^2(M)$ is a nondegenerate $2$-form and
let $\theta\in\Omega^1(M)$. Assume that the pair $(\pi,\xi)$ is
associated with the pair $(\omega,\theta)$. We say that the
pseudo-Riemannian metric $g$ is associated with the pair
$(\omega,\theta)$ if $\sharp_{\omega,\theta}:=\sharp_{\pi,\xi}$ is
an isometry, i.e. if
\begin{equation}\label{met-ass-loc-conf-symp}
g\left(\sharp_{\omega,\theta}(\alpha),\sharp_{\omega,\theta}(\beta)\right)
=g^\ast(\alpha,\beta),
\end{equation}
for every $\alpha ,\beta \in \Omega ^{1}(M)$.

If $\theta =0$, then $\xi=0$ and
$\sharp_{\omega,\theta}=\sharp_{\omega}$, and if $J$ and $J^\ast$
are the fields of endomorphisms defined by the formulae (\ref{J}),
then
$$
\begin{array}{ll}
g(\sharp_{\omega,\theta}(\alpha),\sharp_{\omega,\theta}(\beta))
&= g(\sharp_{\omega}(\alpha),\sharp_{\omega}(\beta)) \\
&= g^\ast(\flat_g(\sharp_{\omega}(\alpha)),\flat_g(\sharp_{\omega}(\beta))) \\
&= g^\ast(J^\ast\alpha,J^\ast\beta),
\end{array}
$$
for any $\alpha ,\beta \in \Omega ^{1}(M)$. Hence, in the case
$\theta=0$, the relation (\ref{met-ass-loc-conf-symp}) is equivalent
to
\begin{equation*}
g^\ast\left( J^\ast\alpha ,J^\ast\beta \right) =
g^\ast(\alpha,\beta).
\end{equation*}
If moreover $g$ is positive definite, this last identity means that
the pair $(\omega,g)$ is an almost Hermitian structure on $M$ and
that $J$ is the associated almost complex structure, i.e., we have
$$
g(JX,JY)=g(X,Y) \qquad\textrm{ and }\qquad \omega(X,Y)=g(JX,Y),
$$
for every $X,Y\in\mathfrak{X}(M)$.

\begin{thm}\label{loc-conf-symp-met-ass-levi-civita}
Assume that $(\omega,\theta)$ is a locally conformally symplectic
structure and that $g$ is an associated metric. We have
$$
\sharp_{\omega,\theta}\left(\mathcal{D}_\alpha\beta\right)=\nabla_{\sharp_{\omega,\theta}(\alpha)}\sharp_{\omega,\theta}(\beta)
$$
for every $\alpha,\beta\in \Omega^1(M)$.
\end{thm}
\begin{proof}
By Propositions \ref{Levi-Civita-triplet} and \ref{omega-theta-alg},
we need only to prove  that $\lambda=\theta$. On one hand, we have $
\sharp_{\pi,\xi}(\theta)=\xi$. On the other hand, for any
$\alpha\in\Omega^1(M)$, we have
$$
\begin{array}{ll}
g(\sharp_{\pi,\xi}(\lambda),\sharp_{\pi,\xi}(\alpha))&
=g(\sharp_g(\lambda),\sharp_g(\alpha))\\
 & =g(\xi,\xi)\alpha(\xi)+g(\xi,J\sharp_g(\alpha)) \\
 & =g(\xi,\xi)\alpha(\xi)+g(\xi,\sharp_\pi(\alpha)) \\
 & =g(\xi,\sharp_{\pi,\xi}(\alpha)).
\end{array}
$$
Since $\sharp_{\pi,\xi}$ is an isometry, hence an isomorphism, then
$\sharp_{\pi,\xi}(\lambda)=\xi$.
\end{proof}

\begin{cor}\label{levi-civita-omega}
Under the same hypotheses of the theorem above, we have
$$
\mathcal{D}\pi(\alpha,\beta,\gamma)=\nabla
\omega(\sharp_{\omega,\theta}(\alpha),\sharp_{\omega,\theta}(\beta),\sharp_{\omega,\theta}(\gamma)).
$$
\end{cor}
\begin{proof}
We have $\omega \left( \xi ,\sharp_{\pi}(\alpha) \right) =-
i_{\sharp_{\pi}(\alpha)}\omega(\xi) =\alpha(\xi)$ and likewise
$\omega(\xi,\sharp_\pi(\beta))=\beta(\xi)$, consequently
\begin{equation}\label{omega-pi-xi}
\omega(\sharp_{\pi,\xi}(\alpha)
,\sharp_{\pi,\xi}(\beta))=\pi(\alpha,\beta).
\end{equation}
It suffices now to compute $\nabla
\omega(\sharp_{\pi,\xi}(\alpha),\sharp_{\pi,\xi}(\beta),\sharp_{\pi,\xi}(\gamma))
$ and use the theorem above.
\end{proof}

\section{Compatibility of the triple $(\pi,\xi,g)$}

\subsection{Definition}

We say that the metric $g$  is compatible with the pair $(\pi,\xi)$
or that the triple $(\pi,\xi,g)$ is compatible if
\begin{equation}\label{COMPATIBILITE2}
\mathcal{D}\pi(\alpha,\beta,\gamma)\!=\! \frac{1}{2}\left(
\gamma(\xi) \pi(\alpha ,\beta)\!-\!\beta(\xi) \pi(\alpha,\gamma)\!
-\! J^\ast\gamma(\xi) g^\ast(\alpha,\beta) \!+\! J^\ast\beta(\xi)
g^\ast(\alpha,\gamma) \right),
\end{equation}
for every $\alpha ,\beta,\gamma \in \Omega ^{1}(M)$. The formula
(\ref{COMPATIBILITE2}) can also be written in the form
\begin{equation}\label{COMPATIBILITE}
\left( \mathcal{D}_{\alpha }J^\ast\right) \beta =\frac{1}{2}\left(
\pi(\alpha,\beta)\flat_g(\xi) - \beta(\xi)J^\ast\alpha +
g^\ast(\alpha,\beta) J^\ast\flat_g(\xi) + J^\ast\beta(\xi)\alpha
\right),
\end{equation}
for any $\alpha ,\beta \in \Omega ^{1}(M)$.

The compatibility in the case $\xi$ is the zero vector field means
that $(M,\pi,g)$ is a pseudo-Riemannian Poisson manifold, and
Riemannian Poisson if moreover the metric $g$ is positive definite,
see \cite{boucetta1,boucetta2}.

\subsection{$\frac{1}{2}$-Kenmotsu manifolds}

Recall, see for instance \cite[\S\,6.6]{blair}, that an almost
contact Riemannian structure $(\Phi,\xi,\eta,g)$ on $M$ is said to
be $\frac{1}{2}$-Kenmotsu if we have
\begin{equation*}
\left( \nabla _{X}\Phi \right)(Y) = \displaystyle\frac{1}{2}\left(
g(\Phi(X),Y) \xi -\eta(Y)\Phi(X)\right),
\end{equation*}
for any $X,Y\in \mathfrak{X}(M)$.

\begin{lem}\label{phi-sharp-pi}
Assume that $(\Phi,\xi,\eta,g)$ is an almost contact
pseudo-Riemannian structure on $M$ and let $\pi$ be the associated
bivector field. If the skew algebroid $(T^\ast
M,\sharp_{\pi,\xi},\left[.,.\right] _{\pi ,\xi }^{g})$ is an almost
Lie algebroid, then
$$
\sharp_{\pi,\xi}\left((\mathcal{D}_{\alpha}
J^\ast)\beta\right)=-\left(\nabla_{\sharp_{\pi,\xi}(\alpha)}\Phi\right)(\sharp_{\pi,\xi}(\beta)),
$$
for every $\alpha,\beta\in\Omega^1(M)$.
\end{lem}
\begin{proof}
By using the formula (\ref{f1}) and the fact that we have
$\sharp_g\circ J^\ast=J\circ \sharp_g$, we deduce that $
\sharp_{\pi,\xi}(J^\ast\alpha) = -\Phi(\sharp_g(J^\ast\alpha))+
\eta(\sharp_g(J^\ast\alpha))\xi=-\Phi(\sharp_\pi(\alpha))=-\Phi(\sharp_{\pi,\xi}(\alpha)).
$ Therefore
\begin{equation}\label{pi-xi-J-ast-Phi}
\sharp_{\pi,\xi}\circ J^\ast=-\Phi \circ \sharp_{\pi,\xi}.
\end{equation}
Hence, with Corollary \ref{riem-presque-contact-levi-civita}, we
have
\begin{eqnarray*}
\sharp_{\pi,\xi}\left( \left( \mathcal{D}_{\alpha }J^\ast\right)
\beta \right) &=&\sharp_{\pi,\xi}\left( \mathcal{D}_{\alpha }\left(
J^\ast\beta \right) \right) -\left( \sharp _{\pi ,\xi }\circ
J^\ast\right) \left( \mathcal{D}_{\alpha
}\beta \right) , \\
&=&\mathcal{\nabla }_{\sharp _{\pi ,\xi }\left( \alpha \right)
}\left( \sharp _{\pi ,\xi }\left( J^\ast\beta \right) \right) +\Phi
\left( \sharp _{\pi
,\xi }\left( \mathcal{D}_{\alpha }\beta \right) \right) , \\
&=&-\mathcal{\nabla }_{\sharp_{\pi,\xi}(\alpha)}\left( \Phi \left(
\sharp_{\pi,\xi}(\beta)\right) \right) +\Phi \left(
\mathcal{\nabla }_{\sharp_{\pi,\xi}(\alpha)}\sharp_{\pi,\xi}(\beta)\right) , \\
&=&-\left( \mathcal{\nabla }_{\sharp_{\pi,\xi}(\alpha)}\Phi
\right)(\sharp_{\pi,\xi}(\beta)).
\end{eqnarray*}
\end{proof}

\begin{prop}
Under the same hypotheses of the above lemma, the compatibility of
the triple $(\pi,\xi,g)$ is equivalent to
$$
(\nabla_X \Phi)(Y)=\dfrac{1}{2}\left(g(\Phi(X),Y)\xi
-\eta(Y)\Phi(X)\right),
$$
for any $X,Y\in \mathfrak{X}(M)$, and if moreover the metric $g$ is
positive definite, then the triple $(\pi,\xi,g)$ is compatible if
and only if the almost contact Riemannian manifold
$(M,\Phi,\xi,\eta,g)$ is $\frac{1}{2}$-Kenmotsu.
\end{prop}
\begin{proof}
Since we have $J^\ast\flat_g(\xi)=\flat_g(J\xi)=-\flat_g(\Phi
\xi)=0$ and
$$
J^\ast\beta(\xi)=J^\ast\beta(\sharp_g(\eta))=\eta(\sharp_g(J^\ast\beta))=\eta(J
\sharp_g(\beta))=-\eta(\Phi(\sharp_g(\beta)))=0,
$$
then the formula (\ref{COMPATIBILITE}) becomes
$$
\left( \mathcal{D}_{\alpha }J^\ast\right) \beta  =\frac{1}{2}\left(
\pi(\alpha,\beta)\eta- \beta(\xi)J^\ast\alpha \right).
$$
Applying $\sharp_{\pi,\xi}$ which by Proposition
\ref{pi-ass-presque-contact} is an isometry and hence an
isomorphism, this last formula is equivalent to
$$
\sharp_{\pi,\xi}\left(\left( \mathcal{D}_{\alpha }J^\ast\right)
\beta \right) =\frac{1}{2}\left(
\pi(\alpha,\beta)\sharp_{\pi,\xi}(\eta) -
\beta(\xi)\sharp_{\pi,\xi}(J^\ast\alpha) \right).
$$
Now, by Formula (\ref{f1}), we have $\sharp_{\pi,\xi}(\eta)=\xi$,
and if we put $X=\sharp_{\pi,\xi}(\alpha)$
$Y=\sharp_{\pi,\xi}(\beta) $, then we have $\beta(\xi)=\eta(Y)$,
also using (\ref{pi-xi-J-ast-Phi}), we have
$\sharp_{\pi,\xi}(J^\ast\alpha)=-\Phi(X)$ and
$$
\begin{array}{ll}
\pi(\alpha,\beta)
& =g(\sharp_g(\alpha),\Phi(\sharp_g(\beta))) \\
& =-g(\sharp_g(\alpha),\sharp_g(J^\ast\beta)) \\
& =-g^\ast(\alpha,J^\ast\beta) \\
& =g(X,\Phi(Y)).
\end{array}
$$
It remains to use the lemma above.
\end{proof}

\begin{thm}
Assume that $(\eta,g)$ is a contact Riemannian structure on $M$ and
let $(\Phi,\xi,\eta,g)$ be the associated almost contact Riemannian
structure. Assume that $(\pi,\xi)$ is the Jacobi structure
associated with the contact form $\eta$. Then the triple
$(\pi,\xi,g)$ is compatible if and only if $(M,\Phi,\xi,\eta,g)$ is
$\frac{1}{2}$-Kenmotsu.
\end{thm}
\begin{proof}
We have proved that $\pi$ is the bivector field of the proposition
above
 and that $\lambda=\eta$, see the proof of Theorem \ref{riemannienne-contact-levi-civita}.
\end{proof}

\subsection{Locally conformally K\"ahler manifolds}

Recall that if $\omega$ is a nondegenerate $2$-form  and $g$
 an associated Riemannian metric, the almost Hermitian structure
$(\omega,g)$ is Hermitian if the associated almost complex structure
is integrable, and K\"ahler if moreover $\omega$ is closed. Recall
also that if $(\omega,g)$ is almost Hermitian, then it is K\"ahler
if and only if the $2$-form $\omega$ is parallel for the Levi-Civita
connection of $g$.

If $(\omega,\theta)$ is a locally conformally symplectic structure
and $(\omega,g)$ a Hermitian structure, we say that the triple
$(\omega,\theta,g)$ is a locally conformally K\"ahler structure.

We shall prove that if $(\omega,\theta)$ is a locally conformally
symplectic structure on $M$ and that $(\pi,\xi)$ is the associated
Jacobi structure, if $g$ is a Riemannian metric associated with
$\omega$ and with $(\omega,\theta)$, the compatibility of the triple
$(\pi,\xi,g)$ induces a locally conformally K\"ahler structure on
$M$.

\begin{lem}
Assume that $\omega\in \Omega^2(M)$ is a nondegenerate differential
$2$-form and let  $\theta\in\Omega^1(M)$. Assume that $(\pi,\xi)$ is
the pair associated with $(\omega,\theta)$. If the pseudo-Riemannian
metric $g$ is associated with the $2$-form $\omega$ and with the
pair $(\omega,\theta)$, then we have
$$
J\circ \sharp_{\pi,\xi}=\sharp_{\pi,\xi}\circ J^\ast.
$$
\end{lem}
\begin{proof}
Since the metric $g$ is assumed to be associated with $\omega$ and
using  (\ref{omega-pi-xi}), we get
$$
g(J\sharp_{\pi,\xi}(\alpha),\sharp_{\pi,\xi}(\beta))=\omega(\sharp_{\pi,\xi}(\alpha),\sharp_{\pi,\xi}(\beta))=\pi(\alpha,\beta)=g^\ast(J^\ast\alpha,\beta),
$$
and since  $g$ is also assumed to be associated with the pair
$(\omega,\theta)$, i.e. $\sharp_{\pi,\xi}$ is an isometry, then
$$
g(J\sharp_{\pi,\xi}(\alpha),\sharp_{\pi,\xi}(\beta))=g(\sharp_{\pi,\xi}(J^\ast\alpha),\sharp_{\pi,\xi}(\beta)).
$$
Finally, since $\sharp_{\pi,\xi}$ is an isometry, hence an
isomorphism, the result follows.
\end{proof}

\begin{thm}
Assume that $(\omega,df)$ is a conformally symplectic structure on
$M$ and that $(\pi,\xi)$ is the associated Jacobi structure. If $g$
is a Riemannian metric associated with $\omega$ and with $(\omega
,df)$, then the triple $(\pi,\xi,g)$ is compatible if and only if
the triple $(\omega,df,g)$ is a conformally K\"ahler structure.
\end{thm}
\begin{proof}
We need to prove that the triple $(\pi,\xi,g)$ is compatible if and
only if the pair  $(e^{f}\omega,e^{f}g)$ is compatible, i.e., if and
only if the $2$-form $e^f\omega$ is parallel with respect to the
Levi-Civita connection $\nabla^f$ associated with the metric
$g^f=e^fg$. As the connections $\nabla$ and $\nabla^f$ are related
by the formula
$$
\nabla _{X}^{f}Y=\nabla _{X}Y+\frac{1}{2}\left(
X(f)Y+Y(f)X-g(X,Y)\textrm{grad}_g f \right),
$$
where $\textrm{grad}_gf=\sharp_g(df)$, we deduce that
$$
\begin{array}{lll}
\nabla^f\omega(X,Y,Z) & = & \nabla\omega(X,Y,Z)-X(f)\omega(Y,Z) \\ & & -\frac{1}{2}Y(f)\omega(X,Z)+\frac{1}{2}Z(f)\omega(X,Y) \\
& & +\frac{1}{2}\left(g(X,Y)\omega(\textrm{grad}_g
f,Z)-g(X,Z)\omega(\textrm{grad}_g f,Y)\right),
\end{array}
$$
and then that
$$
\nabla^f (e^f\omega)(X,Y,Z)=
e^f\left(X(f)\omega(Y,Z)+\nabla^f\omega(X,Y,Z)\right) =e^f
\Lambda_f(X,Y,Z),
$$
where we have set
$$
\begin{array}{lll}
\Lambda_f(X,Y,Z) &=& \nabla\omega(X,Y,Z)-\dfrac{1}{2}\left(Y(f)\omega(X,Z)-Z(f)\omega(X,Y)\right) \\
& & +\frac{1}{2}\left(g(X,Y)\omega(\textrm{grad}_g
f,Z)-g(X,Z)\omega(\textrm{grad}_g f,Y)\right).
 \end{array}
$$
It follows that $\nabla^f(e^f\omega)=0$ if and only if
$\Lambda_f=0$, hence that the pair $(e^f\omega,e^fg)$ is compatible
if and only if
$$
\begin{array}{lll}
\nabla\omega(X,Y,Z) & \!=\! & \dfrac{1}{2}\left(Y(f)\omega(X,Z)- Z(f)\omega(X,Y) -g(X,Y)\omega(\textrm{grad}_g f,Z) \right. \\
& & \left. +g(X,Z)\omega(\textrm{grad}_g f,Y)\right).
\end{array}
$$
Let us prove now that this last identity is equivalent to the
formula (\ref{COMPATIBILITE2}). Let $\alpha ,\beta, \gamma \in
\Omega ^{1}(M)$ be such that $X=\sharp_{\pi,\xi}(\alpha)$,
$Y=\sharp_{\pi,\xi}(\beta)$ and $Z=\sharp_{\pi,\xi}(\gamma)$. By
Corollary \ref{levi-civita-omega}, we have $\nabla
\omega(X,Y,Z)=\mathcal{D}\pi(\alpha,\beta,\gamma)$. On the other
hand, setting $\theta=df$, we have $
Y(f)=\theta(Y)=\theta(\sharp_\pi(\beta))+\beta(\xi)\theta(\xi)=-\beta(\sharp_\pi(\theta))=-\beta(\xi)$
and likewise $Z(f)=-\gamma(\xi)$. Also, by (\ref{omega-pi-xi}), we
have $\omega(X,Y)=\pi(\alpha,\beta)$ and
$\omega(X,Z)=\pi(\alpha,\gamma)$. Finally, since the metric $g$ is
associated with $\omega$, it follows that
$$
\begin{array}{ll}
\omega(\textrm{grad}_g f, Y) & =-\omega(Y,\sharp_g(\theta)) \\
& =-g(JY,\sharp_g(\theta)) \\
& =-\theta(JY) \\
& =\omega(\xi,JY) \\
& =\omega(\sharp_{\pi,\xi}(\theta),J\sharp_{\pi,\xi}(\beta)),
\end{array}
$$
and since $g$ is associated with $\omega$ and with
$(\omega,\theta)$, by using the lemma above and (\ref{omega-pi-xi}),
we get
$$
\omega(\textrm{grad}_g f,
Y)=\omega(\sharp_{\pi,\xi}(\theta),\sharp_{\pi,\xi}(J^\ast\beta))=\pi(\theta,J^\ast\beta)=J^\ast\beta(\sharp_\pi(\theta))=J^\ast\beta(\xi)
$$
and likewise $\omega(\textrm{grad}_g f, Z)=J^\ast\gamma(\xi)$.
\end{proof}



\bigskip
Y. A\"{\i}t Amrane, Laboratoire Alg\`ebre et Th\'eorie des
Nombres,\\
Facult\'e de Math\'ematiques,\\
USTHB, BP 32, El-Alia, 16111 Bab-Ezzouar, Alger, Algeria. \\
e-mail : yacinait@gmail.com \bigskip\\
A. Zeglaoui, Laboratoire Alg\`ebre et Th\'eorie des
Nombres,\\
Facult\'e de Math\'ematiques,\\
USTHB, BP 32, El-Alia, 16111 Bab-Ezzouar, Alger, Algeria. \\
e-mail : ahmed.zeglaoui@gmail.com \\

\end{document}